\documentclass[psamsfonts]{amsart}

\usepackage{amssymb, url, tikz}
\usetikzlibrary{arrows}
\newtheorem{theorem}{Theorem}
\newtheorem{lemma}[theorem]{Lemma}

\newtheorem{corollary}[theorem]{Corollary}
\newtheorem{proposition}[theorem]{Proposition}
\newtheorem{sublemma}[theorem]{Sublemma}

\theoremstyle{remark}

\newcommand{\del}{\partial}
\newcommand{\abs}[1]{\left\lvert #1 \right\rvert}

\DeclareMathOperator{\Vol}{Vol}

\DeclareMathOperator{\dist}{dist}
\DeclareMathOperator{\ind}{ind}
\DeclareMathOperator{\hyp}{hyp}
\DeclareMathOperator{\ess}{ess}
\DeclareMathOperator{\rel}{rel}

\begin{document}

\title[Simplicial volume and Morse broken trajectories]{Using simplicial volume to count maximally broken Morse trajectories}
\author{Hannah Alpert}
\address{MIT\\ Cambridge, MA 02139 USA}
\email{hcalpert@math.mit.edu}
\begin{abstract}
Given a closed Riemannian manifold of dimension $n$ and a Morse-Smale function, there are finitely many $n$-part broken trajectories of the negative gradient flow.  We show that if the manifold admits a hyperbolic metric, then the number of $n$-part broken trajectories is always at least the hyperbolic volume.  The proof combines known theorems in Morse theory with lemmas of Gromov about simplicial volumes of stratified spaces.
\end{abstract}
\maketitle

\section{Introduction}

Let $M$ be a closed manifold of dimension $n$, and let $f : M \rightarrow \mathbb{R}$ be a Morse function.  For any Riemannian metric $g$ on $M$, we consider the flow along the negative gradient vector field $-\nabla f$; for every critical point $p$ of $f$ with Morse index $\ind(p)$, the flow determines the descending (or unstable) manifold $\mathcal{D}(p)$ of dimension $\ind(p)$ and the ascending (or stable) manifold $\mathcal{A}(p)$ of dimension $n - \ind(p)$.  The function $f$ satisfies \textbf{\emph{Morse-Smale transversality}} with respect to the metric $g$ if for every two critical points $p$ and $q$, the manifolds $\mathcal{D}(p)$ and $\mathcal{A}(q)$ have transverse intersection.  In particular, if the indices of $p$ and $q$ differ by $1$, then there is a discrete set of unparametrized flow lines between $p$ and $q$.  An \textbf{\emph{$n$-part broken trajectory}} is a maximum-length path through the resulting directed graph; that is, it is a sequence of critical points $p_n, p_{n-1}, \ldots, p_1, p_0$ where the index of point $p_i$ is $i$, and a sequence of unparametrized flow lines $\gamma_n, \gamma_{n-1}, \ldots, \gamma_1$ where each $\gamma_i$ runs from $p_i$ to $p_{i-1}$.

The main theorem of this paper is stated as follows.

\begin{theorem}\label{main-thm}
Let $M$ be a closed, oriented manifold of dimension $n \geq 2$, admitting a hyperbolic metric.  Let $g$ be an arbitrary Riemannian metric on $M$, and let $f : (M, g) \rightarrow \mathbb{R}$ be a Morse function satisfying Morse-Smale transversality.  Then we have
\[\#(n\text{-part broken trajectories of }{-\nabla f}) \geq \frac{\Vol(M, \hyp)}{\Vol \Delta^n},\]
where $\Vol \Delta^n$ denotes the supremal volume of a straight simplex in $n$-dimensional hyperbolic space.
\end{theorem}

In particular, if $M$ is a closed, oriented manifold such that there exists a Morse-Smale gradient vector field with no $n$-part broken trajectories (for instance, if there is not a critical point of every index), then $M$ cannot admit a hyperbolic metric.

The hyperbolic volume plays the role of a topological invariant of $M$; the larger the hyperbolic volume, the more topologically complex $M$ is.  In even dimension, the Chern-Gauss-Bonnet theorem implies that the hyperbolic volume is proportional to the Euler characteristic and thus to the sum of the Betti numbers.  In odd dimension, the Euler characteristic is zero, but the Mostow rigidity theorem (which applies to every dimension at least 3, odd or even) states that $M$ has at most one hyperbolic metric and so the hyperbolic volume is well-defined.  However, the hyperbolic volume in odd dimension is not proportional to the sum of the Betti numbers; this fact is implied by the following proposition, which we state for contrast with the main theorem.  A similar example appears on page~30 of Gromov's paper~\cite{Gromov09}.  The proof is not related to the rest of this paper, but we include it for completeness.

\begin{proposition}
There is a sequence $M_1, M_2, \ldots$ of closed, oriented hyperbolic manifolds of dimension $3$, and a sequence of Morse functions $f_i : M_i \rightarrow \mathbb{R}$, such that the number of critical points of the Morse functions is uniformly bounded, but the hyperbolic volumes satisfy
\[\lim_{i \rightarrow \infty} \Vol(M_i, \hyp) = \infty.\]
\end{proposition}

\begin{proof}
Let $\Sigma$ be an oriented surface of genus at least $2$, and let $\varphi : \Sigma \rightarrow \Sigma$ be a pseudo-Anosov diffeomorphism, so that the mapping torus $M_1$ is a hyperbolic manifold~\cite{Thurston98}.  Let $M_i$ be the mapping torus of the $i$th iterate $\varphi^i$ of $\varphi$.  Then $M_i$ is an $i$-fold cover of $M_1$, and therefore has hyperbolic volume given by
\[\Vol(M_i, \hyp) = i \cdot \Vol(M_1, \hyp).\]

To create a Morse function on $M_i$ with few critical points, we view $M_i$ as two copies of $\Sigma \times [0, 1]$ glued along the boundary.  We find a Morse function $f$ on $\Sigma \times [0, 1]$ that has the boundary as a level set and is a standard projection near the boundary; that is, for some $\varepsilon > 0$, for all $(x, t) \in \Sigma \times [0, \varepsilon)$ we have $f(x, t) = t$ and for all $(x, t) \in \Sigma \times (1-\varepsilon, 1]$ we have $f(x, t) = 1-t$.  To construct $f_i$ we take $f$ on one copy of $\Sigma \times [0, 1]$, and $-f$ on the other copy, and glue them together.
\end{proof}

Because the Morse inequalities state that the number of Morse critical points is always at least the sum of the Betti numbers, in this example the Betti numbers are bounded while the hyperbolic volume is unbounded.  The main theorem says that even though the number of critical points is fixed, there must be an increasing number of flow lines between them in order to have an unbounded number of $3$-part broken trajectories.

The main theorem fits into a family of theorems initiated by Gromov in the paper~\cite{Gromov09}.  There, he shows that the hyperbolic volume (really the Gromov simplicial norm; the definition appears in Section~\ref{simplicial} of this paper) is a lower bound for a different count describing the complexity of a map: he considers a smooth map of manifolds $X^n \rightarrow Y^{n-1}$ that drops in dimension by $1$, and counts the number of points at which the set of singular values---an immersed space in $Y$ of codimension $1$---intersects itself with the maximum possible multiplicity.  The paper~\cite{Alpert15} uses the same methods to show that the simplicial norm is a lower bound for again another count: we consider a nonvanishing gradient vector field on a manifold with boundary, and count the flow lines that have a maximum number of tangencies to the boundary.  In this paper we bring those methods to the setting of Morse-Smale gradient vector fields.

In the proof of the main theorem, we view the manifold $M$ as a CW-complex in which the cells are the descending manifolds of the critical points.  In Section~\ref{simplicial} we use tools from Gromov's paper~\cite{Gromov09} to bound the hyperbolic volume of $M$ in terms of the CW structure.  Then in Section~\ref{morse}, we invoke machinery from Morse theory in order to justify the needed properties of the CW structure, including its relationship to the $n$-part broken trajectories.  We conclude in Section~\ref{end} by finishing the proof of the main theorem.

\emph{Acknowledgments.}  The main theorem was conjectured by Gabriel Katz related to our collaboration on the paper~\cite{Alpert15}.  I would like to thank him and my advisor Larry Guth for conversations about that project.  I would also like to thank Josh Greene, Jon Bloom, and Nati Blaier for pointing me toward the appropriate references.

\section{Simplicial volume of stratified spaces}\label{simplicial}

The mathematics in this section is all based on Gromov's simplicial norm, which was introduced in~\cite{Gromov82} and is a function on singular homology classes with real coefficients.  The simplicial norm is defined as follows.  For every singular chain $c$ on a topological space $X$, the norm of $c$, denoted $\Vert c \Vert_\Delta$, is the sum of absolute values of the (real) coefficients.  For every real homology class $h$, the \textbf{\emph{simplicial norm}} of $h$, denoted $\Vert h \Vert_{\Delta}$, is the infimum of $\Vert c \Vert_\Delta$ over all cycles $c$ representing $h$.  The simplicial norm is often called the \textbf{\emph{simplicial volume}} because it generalizes hyperbolic volume: if $M$ is any closed, oriented hyperbolic manifold of dimension $n$, with fundamental class $[M]$, then the simplicial norm is related to the hyperbolic volume by the formula
\[\Vert [M] \Vert_{\Delta} = \frac{\Vol(M, \hyp)}{\Vol \Delta^n},\]
where $\Vol \Delta^n$ denotes the supremal volume of a straight simplex in $n$-dimensional hyperbolic space (Proportionality Theorem, p.~11 of~\cite{Gromov82}).

Because of this proportionality, our main theorem has an equivalent formulation in terms of the simplicial volume instead of the normalized hyperbolic volume.  The revised statement, which appears as Theorem~\ref{general}, applies not only to manifolds admitting a hyperbolic metric but to any closed manifolds with nonzero simplicial volume.  Among these are manifolds with sectional curvature pinched between two negative constants; products of hyperbolic manifolds, which never admit a hyperbolic metric; and, generalizing the products, locally symmetric spaces of non-compact type, as proved by Lafont and Schmidt in~\cite{Lafont06}.

In the paper~\cite{Gromov09}, Gromov introduces the Amenable Reduction Lemma and the Localization Lemma in order to bound simplicial norms in stratified spaces; see~\cite{Alpert15} for a more detailed exposition of those lemmas.  The purpose of this section is to prove a version of the Localization Lemma that applies to CW-complexes, which have an obvious partition by the open cells of the CW structure; this partition does not satisfy the definition of stratification but behaves like a stratification for this purpose.

In order to state this variant of the Localization Lemma, we need to define generalized stratifications and a stratified version of the simplicial norm.  In Gromov's paper~\cite{Gromov09}, a \textbf{\emph{stratification}} of a space is defined to be any partition with the following property: if a stratum $S$ intersects the closure $\overline{S'}$ of another stratum $S'$, then $S \subseteq \overline{S'}$ and we write $S \preceq S'$.  In this paper, we want the theorems about stratifications to apply to CW-complexes, partitioned by the open cells; however, not all CW-complexes are stratifications by this definition.  Therefore, we define a \textbf{\emph{generalized stratification}} of a space to be any partition with the following property.  If a stratum $S$ intersects the closure $\overline{S'}$ of another stratum $S'$, then we write $S' \rightarrow S$, pronounced as ``$S'$ limits to $S$''.  In order for the partition to be a generalized stratification, we require that the resulting directed graph on the set of strata have no directed cycles.  In this case, the directed graph generates a partial order on strata: we write $S \preceq S'$ whenever there is a directed path from $S'$ to $S$.  If neither $S \preceq S'$ nor $S' \preceq S$, then we say the two strata are \textbf{\emph{incomparable}}.  In this paper we usually assume that strata are connected; given an arbitrary generalized stratification, the set of connected components of strata is also a generalized stratification.

The stratified simplicial norm, like the simplicial norm, is an infimum of sums of coefficients, but the infimum is taken over only those cycles $c$ that are consistent with the (generalized) stratification, in the following sense:
\begin{itemize}
\item The \emph{cellular} condition requires that for each simplex of $c$, the image of the interior of each face (of any dimension) must be contained in one stratum.
\item The \emph{order} condition requires that the image of each simplex of $c$ must be contained in a totally ordered chain of strata; that is, the simplex does not intersect any two incomparable strata.  
\item The \emph{internality} condition requires that for each simplex of $c$, if the boundary of a face (of any dimension) maps into a stratum $S$, then the whole face maps into $S$.  
\end{itemize}
For technical reasons, we also require a fourth condition, which is stated in terms a space $\Sigma$ that is constructed from $c$ as follows.  Let $j$ be the dimension of $c$, and let $\Delta^j$ denote the abstract $j$-simplex.  For each singular simplex $\sigma_i : \Delta^j \rightarrow X$ appearing in $c$, there are $j+1$ face maps from $\Delta^{j-1}$ to $X$ obtained by restricting $\sigma_i$.  We form $\Sigma$ by taking one copy of $\Delta^j$ for each $\sigma_i$ and identifying the faces that have the same face map.  (A similar construction appears on pages 108--109 of Hatcher's textbook~\cite{Hatcher02}.)  Note that every face must be glued to at least one other face, because otherwise it would appear with nonzero coefficient in the linear combination $\del c$ of face maps, contradicting the cycle hypothesis $\del c = 0$.  Then we can view $c$ as a triple $(\Sigma, c_\Sigma, \sigma)$, where $c_\Sigma$ is a simplicial cycle on $\Sigma$, and $\sigma : \Sigma \rightarrow X$ is a continuous map such that $c = \sigma_*c_\Sigma$.  Note that the space $\Sigma$ is not necessarily an honest simplicial complex, linearly embeddable in Euclidean space, but is what Hatcher calls a $\Delta$-complex, which may have (for instance) edges in its $1$-skeleton that are self-loops (i.e., both endpoints are the same vertex).   The fourth condition is the following.
\begin{itemize}
\item The \emph{loop} condition requires that for every edge in the $1$-skeleton of $\Sigma$ that is a self-loop, its image in $X$ must be a single point.
\end{itemize}

In the papers \cite{Gromov09} and \cite{Alpert15}, the stratified simplicial norm of a homology class is the infimal sum of coefficients of cycles that represent the homology class and satisfy the conditions above.  In this paper, we only count the simplices for which every vertex is in a different stratum; these are called essential simplices, a term which is defined in a slightly more general context later in the paper.  For any cycle $c = \sum r_i \sigma_i$ (where $\sigma_i$ are singular simplices and $r_i$ are real coefficients) and generalized stratification $\mathcal{S}$, let $\Vert c \Vert_{\Delta, \ess}^{\mathcal{S}}$ be defined by
\[\Vert c \Vert_{\Delta, \ess}^{\mathcal{S}} = \sum_{\text{essential }\sigma_i} \abs{r_i}.\]

The \textbf{\emph{essential stratified simplicial norm}} of a homology class $h$ with respect to a generalized stratification $\mathcal{S}$, denoted $\Vert h \Vert_{\Delta, \ess}^{\mathcal{S}}$, is the infimum of $\Vert c \Vert_{\Delta, \ess}^{\mathcal{S}}$ taken over all cycles $c$ representing $h$ that satisfy the \emph{cellular}, \emph{order}, \emph{internality}, and \emph{loop} conditions.  When $h$ is a relative homology class, the essential stratified simplicial norm is defined in the same way, as an infimum over relative cycles.

The goal of this section is the following lemma.

\begin{lemma}\label{simplicial-goal}
Let $X$ be a finite CW-complex.  For any homology class $h \in H_j(X)$, let $h_{\rel}$ denote the corresponding relative homology class in $H_j(X, X^{j-1})$.  Let $\mathcal{S}$ denote the generalized stratification of $X$ consisting of the open cells of the CW structure.  Let $Z$ be any topological space such that the universal cover is contractible, and let $\alpha : X \rightarrow Z$ be a continuous map.  Then the simplicial norm of the class $\alpha_*h \in H_j(Z)$ satisfies the bound
\[\Vert \alpha_* h \Vert_{\Delta} \leq \Vert h_{\rel} \Vert_{\Delta, \ess}^{\mathcal{S}}.\]
\end{lemma}

To prove this Lemma~\ref{simplicial-goal}, we prove the Generalized Localization Lemma (Lemma~\ref{gen-loc}), for which Lemma~\ref{simplicial-goal} and the Localization Lemma of~\cite{Gromov09} are special cases.  The proof is very similar to the proof of the Localization Lemma (given with more detail in~\cite{Alpert15}), and it relies on the Amenable Reduction Lemma of~\cite{Gromov09}.  There is one main way in which the proof of the Generalized Localization Lemma goes beyond what is needed for Lemma~\ref{simplicial-goal}.  A hypothesis involving amenable groups appears in the Amenable Reduction Lemma and the Generalized Localization Lemma because it is a hypothesis in the Localization Lemma; however, in Lemma~\ref{simplicial-goal} these amenable groups are zero because they are the fundamental groups of the cells of the CW-complex.  Thus, for the proofs of Lemma~\ref{simplicial-goal} and the main theorem of this paper, it suffices to substitute ``zero'' wherever ``amenable group'' appears throughout this section.

In order to state the Amenable Reduction Lemma, we need to define the notion of partial coloring.  Let $c$ be a cycle, viewed as the triple $(\Sigma, c_{\Sigma}, \sigma)$.  By a \textbf{\emph{partial coloring}} of $c$ we mean a list $V_1, V_2, \ldots, V_\ell, \ldots$ of disjoint subsets of the set of vertices of $\Sigma$.   According to the partial coloring we classify each simplex as either essential or non-essential; the non-essential simplices are the ones that can be made to disappear in a certain sense.  A simplex $\Delta$ of $\Sigma$ is a non-essential simplex of $c$ if either of the following conditions holds:
\begin{itemize}
\item $\Delta$ has two distinct vertices in the same $V_\ell$ (the vertices are permitted to have the same image in $X$ as long as they are distinct in $\Sigma$); or
\item $\Delta$ has two vertices that are the same point of $\Sigma$, and the edge between them is a null-homotopic loop in $X$.
\end{itemize}
An \textbf{\emph{essential simplex}} of $c$ is any simplex $\Delta$ of $\Sigma$ that is not non-essential; that is, $\Delta$ is essential if in the $1$-skeleton of $\Delta$ in $\Sigma$, every vertex is in a different $V_\ell$, and any edges that are self-loops map to non-contractible loops in $X$.  In particular, any simplex $\Delta$ with vertices in $\dim(\Delta) + 1$ different sets $V_\ell$ is essential.  (In the definition of essential stratified simplicial norm, the essential simplices were defined according to the partial coloring by strata; there, the \emph{loop} condition replaces the second condition for essential simplex.)  

Let $Z$ be a space with contractible universal cover and let $\alpha : X \rightarrow Z$ be a continuous map that sends all vertices of $c$ in $X$ to the same point of $Z$.  Let $\Gamma_\ell$ denote the subgroup of $\pi_1(Z)$ generated by the $\alpha$-images of the edges of $c$ for which both endpoints are in $V_\ell$.

\begin{lemma}[Amenable Reduction Lemma, p.~25 of~\cite{Gromov09}]\label{amen-red}
Let $c$ be a cycle on $X$ with a partial coloring $\{V_\ell\}$, let $Z$ and $\alpha : X \rightarrow Z$ be as above, and suppose that $\Gamma_\ell$ is an amenable group for every $\ell$.  Then the simplicial norm of the $\alpha$-image of the homology class $[c] \in H_*(X)$ represented by the cycle $c = \sum r_i \sigma_i$ (where $r_i \in \mathbb{R}$ are coefficients and $\sigma_i$ are simplices) satisfies the bound
\[\Vert \alpha_*[c]\Vert_{\Delta} \leq \sum_{\text{essential }\sigma_i} \abs{r_i}.\]
\end{lemma}

It is easier to use a more general version of this lemma, obtained by removing the assumption that $\alpha : X \rightarrow Z$ sends every vertex of $c$ to the same point.  Without this assumption, $\Gamma_\ell$ is not a well-defined subgroup of $\pi_1(Z)$.  Instead, for each $\ell$ let $\Sigma_\ell^1$ denote the 1-dimensional subcomplex of $\Sigma$ consisting of the edges for which both endpoints are in $V_\ell$.  Instead of assuming that $\Gamma_\ell$ is an amenable group, we assume that for each connected component of each $\Sigma_\ell^1$, the image of its fundamental group under the composition of the maps $\sigma_*: \pi_1(\Sigma_\ell^1) \rightarrow \pi_1(X)$ and $\alpha_*: \pi_1(X) \rightarrow \pi_1(Z)$ is an amenable group.  (This image is well-defined only up to conjugation, because base points are not specified.)

\begin{corollary}\label{amen-red-cor}
Let $c$ be a cycle on space $X$ with partial coloring $\{V_\ell\}$.  Suppose that $Z$ is a space with contractible universal cover, and let $\alpha : X \rightarrow Z$ be a continuous map such that for every connected component of every induced 1-complex $\Sigma^1_\ell$ of $c$, the $\alpha$-image of the fundamental group in $\pi_1(Z)$ is an amenable group.  Then the simplicial norm of the $\alpha$-image of the homology class $[c] \in H_*(X)$ represented by the cycle $c = \sum r_i \sigma_i$ (where $r_i \in \mathbb{R}$ are coefficients and $\sigma_i$ are simplices) satisfies the bound
\[\Vert \alpha_*[c]\Vert_{\Delta} \leq \sum_{\sigma_i \textrm{ essential}} \abs{r_i}.\]
\end{corollary}

\begin{proof}
Without loss of generality we may assume that every induced $1$-complex $\Sigma_\ell^1$ is connected; this is because if we subdivide $V_\ell$ according to the connected components of $\Sigma_\ell^1$, there is no change to which simplices are essential.  

We view $c$ as the triple $(\Sigma, c_\Sigma, \sigma)$ and homotope the map $\sigma : \Sigma \rightarrow X$ to a new map $\sigma'$ so that every vertex of $\Sigma$ moves to the same point of $X$ and the resulting group $\Gamma_\ell$ generated by the $\alpha \circ \sigma'$-images of the edges in $\Sigma_\ell^1$ is a subgroup of $\alpha_*\sigma_*\pi_1(\Sigma_\ell^1)$, which we have assumed is amenable.  Specifically, we choose one base point $*$ in $X$ and one base point $*_\ell$ in each $\Sigma_\ell^1$, and choose a path $\gamma_\ell$ in $X$ from $\sigma(*_\ell)$ to $*$.  Then for each vertex $v$ in $\Sigma_\ell^1$ we choose a path in $\Sigma_\ell^1$ from $v$ to $*_\ell$.  We homotope $\sigma$ in a small neighborhood of $v$ in $\Sigma$ so that $v$ travels along the image in $X$ of that path to $\sigma(*_\ell)$, and then along $\gamma_\ell$ to $*$.  We apply this sequence of homotopies to $\sigma$, one for each vertex, to get some new map $\sigma'$.

For every edge in $\Sigma_\ell^1$, the restriction of $\sigma'$ to that edge is a loop in $X$ which is obtained by taking the $\sigma$-image of a loop in $\Sigma_\ell^1$ and conjugating by the path $\gamma_\ell$.  Thus, the subgroup $\Gamma_\ell$ resulting from $\sigma'$ is a subgroup of the conjugation of $\alpha_*\sigma_*\pi_1(\Sigma_\ell^1, *_\ell)$ by $\alpha \circ \gamma_\ell$.  We have assumed that $\alpha_*\sigma_*\pi_1(\Sigma_\ell^1)$ is an amenable group, and every subgroup of an amenable group is amenable, so $\Gamma_\ell$ is amenable.  We may now apply the Amenable Reduction Lemma (Lemma~\ref{amen-red}) to the cycle $c' = (\Sigma, c_\Sigma, \sigma')$.

\end{proof}

In~\cite{Alpert15} when proving the Localization Lemma from~\cite{Gromov09}, we assume that the ambient space is a manifold and the strata are submanifolds; in that case the strata have tubular neighborhoods.  To generalize the lemma to CW-complexes, we define the NDR property, which generalizes the tubular neighborhood property and is satisfied by CW-complexes.  We say that a generalized stratification on a space $X$ has the \textbf{\emph{NDR property}} if every stratum $S$ is a neighborhood deformation retract, in the following sense: there is a neighborhood $U_S$ of $S$ in $X$ and a homotopy $H : U_S \times [0, 1] \rightarrow X$ such that the start map $H_0 : U_S \rightarrow X$ is the inclusion, the image $H_1(U_S)$ of the end map is contained in $S$, and the restriction $H_t\vert_S : S \rightarrow X$ to $S$ at every time $t$ is the inclusion.

\begin{lemma}[Generalized Localization Lemma]\label{gen-loc}
Let $X$ be a compact metrizable space, and let $\mathcal{S}$ be a generalized stratification on $X$ with the NDR property, such that $\mathcal{S}$ has finitely many strata and each stratum is connected.  Let $Z$ be a space with contractible universal cover, and suppose $\alpha : X \rightarrow Z$ is a continuous map such that for every stratum $S$, the group $\alpha_*\pi_1(S) \subseteq \pi_1(Z)$ is amenable.  Let $h \in H_j(X)$ be a homology class, and suppose that $A$ is a closed subset of $X$ such that among the strata intersecting $A$, there is no totally ordered chain of more than $j$ strata.  Let $h_{\rel}$ be the image of $h$ in $H_j(X, A)$.  Then the simplicial norm of $\alpha_*h \in H_j(Z)$ satisfies the bound
\[\Vert \alpha_*h \Vert_{\Delta} \leq \Vert h_{\rel} \Vert_{\Delta, \ess}^{\mathcal{S}}.\]
\end{lemma}

The idea of the proof is to take any relative cycle representing $h_{\rel} \in H_j(X, A)$ and extend it to a cycle representing $h$, by adding a chain in $A$ with no essential simplices with respect to a partial coloring that comes from the stratification.  Then we apply Corollary~\ref{amen-red-cor} to show that the new non-essential simplices do not contribute to the simplicial norm of $\alpha_* h$, thus finishing the proof.  The following sublemma states that any representative of $h_{\rel}$ may be extended by a chain in $A$ to a representative of $h$.

\begin{sublemma}\label{les}
Let $(X, A)$ be a pair of spaces, let $h \in H_j(X)$ be a homology class, and let $h_{\rel} \in H_j(X, A)$ be the corresponding relative homology class.  Let $c_{\rel}$ be any relative cycle such that $[c_{\rel}] = h_{\rel}$.  Then there exists a chain $c_A$ in $A$ such that $\del c_A = -\del c_{\rel}$ (so $c_{\rel} + c_A$ is a cycle in $X$) and $[c_{\rel} + c_A] = h$.
\end{sublemma}

\begin{proof}
The proof is by chasing the long exact sequence of homology for the pair $(X, A)$.  By the exactness at $H_j(X, A)$, we see that the class of $\del c_{\rel}$ must be zero in $H_{j-1}(A)$, so there is some chain $c_1$ in $A$ such that $\del c_1 = \del c_{\rel}$.  By the exactness at $H_j(X)$, because $h - [c_{\rel} - c_1] \in H_j(X)$ maps to zero in $H_j(X, A)$ there is some cycle $c_2$ in $A$ such that $[c_2] = h - [c_{\rel} - c_1]$.  We set $c_A = -c_1 + c_2$.
\end{proof}

In order to construct a partial coloring on some chain $c_A$ in $A$ such that there are no essential simplices, we modify the strata to create a related partition of $A$, and then color the vertices of $c_A$ according to which set of the partition they are in.  In the following sublemma we construct this partition of $A$.

\begin{sublemma}\label{partition}
Let $X$ be a compact metric space, with a generalized stratification satisfying the NDR property, such that there are finitely many strata.  Let $A$ be any closed subset of $X$, and let $\varepsilon > 0$ be arbitrary.  Then there exists a partition $\{P_S\}$ of $A$, containing one subset $P_S$ for each stratum $S$, and $\delta > 0$ such that the following properties hold:
\begin{itemize}
\item For every stratum $S$, the set $P_S$ is contained in the $\varepsilon$-neighborhood of $S$.
\item If $S$ and $S'$ are incomparable strata, then the distances $\dist(P_S, P_{S'})$, $\dist(P_S, S')$, and $\dist(S, P_{S'})$ are all greater than $\delta$.
\item For every stratum $S$, the $\delta$-neighborhood of $P_S$ in $A$ is contained in an NDR neighborhood $U_S$ of $S$ in $X$.
\end{itemize}
\end{sublemma}

\begin{proof}
We start by numbering the strata $S_1, \ldots, S_r$ such that if $S_j \preceq S_i$ then $j \leq i$.  We construct $P_{S_i}$ in order, with a recursion assumption that if $P_{S_j}$ have been constructed for all $j < i$, then their union $\bigcup_{j < i} P_{S_j}$ is an $A$-open neighborhood of the (closed) union $A \cap \left(\bigcup_{j < i} S_j \right)$.  To construct $P_{S_i}$, we let $K_i$ be the complement of $\bigcup_{j < i} P_{S_j}$ in $A \cap S_i$; the set $K_i$ is compact.  If $S'$ is any stratum so that its closure $\overline{S'}$ is disjoint from $S_i$ (equivalently $S' \not \rightarrow S_i$), then $\overline{S'}$ has positive distance from $K_i$.  We select $\varepsilon_i < \varepsilon$ such that $3\varepsilon_i$ is less than this positive distance for all such strata $S'$ and such that the $3\varepsilon_i$-neighborhood $N_{3\varepsilon_i}(K_i)$ (taken in $A$ or $X$, it doesn't matter) is contained in an NDR neighborhood $U_{S_i}$ of $S_i$ in $X$.  We let $P_{S_i}$ be the complement of $\bigcup_{j < i}P_{S_j}$ in $A \cap N_{\varepsilon_i}(K_i)$.  The recursion assumption is preserved: the new union $\bigcup_{j \leq i} P_{S_j}$ is an $A$-open set containing $A \cap \left(\bigcup_{j \leq i} S_j \right)$.

Having constructed all $P_{S_i}$, we set $\delta < \min \varepsilon_i$ and check that the construction satisfies the three conclusions of the lemma.  The first conclusion is true because $P_{S_i}$ is in the $\varepsilon_i$-neighborhood of $S_i$.  For the second conclusion, suppose that $S_i$ and $S_j$ are incomparable strata.  Because $S_j \not\rightarrow S_i$, all of $S_j$ lies further than $3\varepsilon_i$ away from $K_i$, whereas all of $P_{S_i}$ lies within $\varepsilon_i$ of $K_i$.  Thus we have $\dist(P_{S_i}, S_j) > 2\varepsilon_i > \delta$ and similarly $\dist(P_{S_i}, P_{S_j}) > 2\varepsilon_i - \varepsilon_j$.   We may reverse the roles of $S_i$ and $S_j$ if necessary, to assume $\varepsilon_i \geq \varepsilon_j$, in which case $2\varepsilon_i - \varepsilon_j > \delta$.  The third conclusion is true because $N_{\delta}(P_{S_i}) \subseteq N_{3\varepsilon}(K_i) \subseteq U_{S_i}$.
\end{proof}

We are almost ready for the proof of the Generalized Localization Lemma (Lemma~\ref{gen-loc}).  First we record an easy sublemma that says that the \emph{cellular}, \emph{order}, \emph{internality}, and \emph{loop} properties are preserved by barycentric subdivision.

\begin{sublemma}\label{bary}
Let $X$ be a space with a generalized stratification, and let $c = (\Sigma, c_\Sigma, \sigma)$ be a cycle satisfying the \emph{cellular} property.  Then the first barycentric subdivision of $c$ satisfies the \emph{cellular}, \emph{order}, \emph{internality}, and \emph{loop} properties.
\end{sublemma}

\begin{proof}
The \emph{cellular} property is immediate because every open face of the subdivision is contained in an open face of $c$.  The \emph{order} property is proven as follows.  In any simplex of $c$, if $f$ is a face in the boundary of another face $f'$, and $S$ and $S'$ are the strata of $f$ and $f'$, then those strata are related by $S' \rightarrow S$.  Then in the barycentric subdivision, every simplex is contained in a totally ordered chain of faces, so the corresponding strata are also totally ordered.  The \emph{internality} property holds because in any face of the subdivision, the stratum of the interior also appears as the stratum of one of the vertices (at least).  The \emph{loop} property holds because the barycentric subdivision of any $\Sigma$ is an honest simplicial complex with no loops.
\end{proof}

\begin{proof}[Proof of the Generalized Localization Lemma (Lemma~\ref{gen-loc})]
Let $c_{\rel}$ be a relative $(X, A)$ cycle that represents $h_{\rel}$ and satisfies the \emph{cellular}, \emph{order}, \emph{internality}, and \emph{loop} conditions.  We apply Sublemma~\ref{les} to get a chain $c_A$ in $A$ such that $c_{\rel} + c_A$ is a cycle representing $h$.  We apply Sublemma~\ref{partition} to get a partition $\{P_S\}$ of $A$ and a number $\delta > 0$.  We construct, as input to Corollary~\ref{amen-red-cor}, a cycle $c = c_{\rel} + c_1 + c_2 + c_{\delta}$ representing $h$ and a partial coloring on $c$ with one subset of vertices $V_\ell$ for each stratum $S_\ell$, as follows:
\begin{itemize}
\item To construct $c_\delta$, we start with $c_A$ and apply iterated barycentric subdivision until the diameter of each simplex in $X$ is less than $\delta$.  For each vertex $v$ of $c_\delta$, if $v \in P_{S_\ell}$, then $v \in V_\ell$.
\item $c_2$ is the cylinder $-\del c_\delta \times [0, 1]$, triangulated in such a way that no new vertices are created, and mapped to $X$ by the projection $-\del c_\delta \times [0, 1] \rightarrow -\del c_\delta$.  The vertices of $-\del c_\delta \times 1$ are identified with the vertices of $c_\delta$, so their partial coloring is determined by their membership in $P_{S_\ell}$.  The vertices of $-\del c_\delta \times 0$ have a different partial coloring: if $v \in S_\ell$, then $v \in V_\ell$.
\item $c_1$ is a subdivision of the cylinder $\del c_{\rel} \times [0, 1]$, mapped to $X$ by the projection $\del c_{\rel} \times [0, 1] \rightarrow \del c_{\rel}$.  The end $\del c_{\rel} \times 0$ is identified with $\del c_{\rel}$ and is not subdivided.  The end $\del c_{\rel} \times 1$ is divided by barycentric subdivision so that it may be identified with the $0$ end of $c_2$, which is equal to $-\del c_\delta$.  The middle of the cylinder is subdivided by concatenating the chain homotopies corresponding to barycentric subdivision, one for each iteration.  For each vertex $v$ of $c_1$, if $v \in S_\ell$, then $v \in V_\ell$.
\item For every vertex $v$ of $c_{\rel}$, if $v \in S_\ell$, then $v \in V_\ell$.
\end{itemize}

First we verify that every simplex in $c_1$, $c_2$, and $c_\delta$ is not essential.  Any loops in any of these chains are null-homotopic because of either the \emph{loop} property on $c_{\rel}$ or the barycentric subdivision as in Sublemma~\ref{bary}.  To show that two of the $j+1$ vertices of each simplex are labeled with the same subset $V_\ell$, we claim that for every two vertices of a simplex, the strata $S_\ell$ corresponding to their labels must be comparable strata.  In $c_1$ this is true by the \emph{order} property or Sublemma~\ref{bary}.  In $c_\delta$ this is true because $\dist(P_S, P_{S'}) > \delta$ whenever $S$ and $S'$ are incomparable strata.  In $c_2$, to show that every two vertices in a simplex have labels corresponding to comparable strata, there are three cases: if both vertices are in $c_1$, we use the argument for $c_1$; if both vertices are in $c_\delta$, we use the argument for $c_\delta$; and if there is one vertex in each, we use the fact that $\dist(S, P_{S'}) > \delta$ whenever $S$ and $S'$ are incomparable strata.  Thus, for every simplex in $c_1$, $c_2$, and $c_\delta$, the $j+1$ vertex labels correspond to a totally ordered chain of strata intersecting $A$, and every such chain has length at most $j$ so two of the labels must be the same.

Let $(\Sigma, c_\Sigma, \sigma)$ be the simplicial complex structure of $c$.  Next we show that for every $V_\ell$, the induced 1-complex $\Sigma_\ell^1$ in $\Sigma$ is homotopic in $X$ to something with image in $S_\ell$.  Then the $\alpha$-image of the fundamental group of every connected component of $\Sigma_\ell^1$ is a subgroup of the $\alpha$-image of the fundamental group of $S_\ell$, which we have assumed to be an amenable group; any subgroup of an amenable group is also amenable.  We homotope $\Sigma_\ell^1$ into $S_\ell$ by showing that it is included in the NDR neighborhood $U_{S_\ell}$, which homotopes into $S_\ell$.  Indeed, every edge in $\Sigma_\ell^1$ is in $S_\ell \cup N_{\delta}(P_{S_\ell})$, where $N_\delta$ denotes the $\delta$-neighborhood: if the edge is in $c_{\rel}$ or $c_1$, then both endpoints are in $S_\ell$, so the edge is in $S_\ell$ by the \emph{internality} condition or Sublemma~\ref{bary}; if the edge is in $c_2$ or $c_{\delta}$ but not in $c_1$, then at least one endpoint is in $P_{S_{\ell}}$, and the whole simplex has diameter less than $\delta$, so the edge is in $N_\delta(P_{S_\ell})$.  So we have 
\[\Sigma_\ell^1 \subseteq S_\ell \cup N_{\delta}(P_{S_\ell}) \subseteq U_{S_\ell},\]
and thus the whole 1-complex may be homotoped into $S_\ell$.

Now every hypothesis of Corollary~\ref{amen-red-cor} is satisfied, so we apply that corollary to obtain
\[\Vert \alpha_*h \Vert_\Delta \leq \Vert c_{\rel} \Vert_{\Delta, \ess}^{\mathcal{S}}.\]
\end{proof}

\begin{proof}[Proof of Lemma~\ref{simplicial-goal}]
First we observe that the generalized stratification of a CW-complex by open cells satisfies the NDR property: every skeleton $X^i$ is a neighborhood deformation retract in the next skeleton $X^{i+1}$, so if $S$ is a cell, we obtain the NDR neighborhood $U_S$ by taking the preimage of $S$ under the composition of all such retractions that have $i$ at least the dimension of $S$.

Every stratum is contractible, so its fundamental group is zero and hence amenable.  We take $A = X^{j-1}$, in which there is no totally ordered chain of more than $j$ strata.  Applying the Generalized Localization Lemma (Lemma~\ref{gen-loc}) finishes the proof.
\end{proof}

\section{Morse theory machinery}\label{morse}

We begin by briefly summarizing the significance of the Morse-Smale condition; all of this information can be found in the book~\cite{Banyaga04}.  An example of a Morse function that is not Morse-Smale is obtained by taking the height function of a torus $T^2$ stood on end, as in Figure~\ref{tori}.  The pair of index-$1$ critical points violates the transversality condition, because there is flow between them: the descending manifold of the upper point and the ascending manifold of the lower point both have dimension $1$, and their intersection in the $2$-dimensional surface has dimension $1$, so it is not transverse.  More generally, if $\mathcal{D}(p)$ and $\mathcal{A}(p)$ have a transverse and non-empty intersection, then the intersection has dimension $\ind(p) - \ind(q)$, which in particular must be strictly positive because the intersection contains $1$-dimensional flow lines.  Thus, in a Morse-Smale flow, for every flow line the starting critical point must have index strictly greater than the ending critical point.

If we tilt the torus slightly, then the height function becomes Morse-Smale.  In fact, every gradient vector field may be approximated in $C^\infty$ by a Morse-Smale gradient vector field (this is called the Kupka-Smale theorem), and every $C^1$-small perturbation of a Morse-Smale gradient vector field is still Morse-Smale (this is a theorem of Palis).  On the tilted torus, each of the index-$1$ critical points has two flow lines up to the index-$2$ critical point and two flow lines down to the index-$0$ critical point, as shown in Figure~\ref{tori}, giving eight $2$-part broken trajectories in all.  (Of course, the main theorem says nothing about this example because the torus does not admit a hyperbolic metric.)  It is not obvious in general that there are only finitely many flow lines between critical points of index difference $1$; this theorem is part of a larger family of compactness results, which say that the space of flow lines can be compactified by adding in the broken trajectories (with some appropriate notion of convergence).  The usual reason to count flow lines between critical points of index difference $1$ is because they define the differential in the Morse-Smale-Witten chain complex, which is a chain complex in which the set of $k$-chains is freely generated by the critical points of index $k$.  The homology of the Morse-Smale-Witten complex is called Morse homology, and it is isomorphic to the singular homology of the manifold.  In that setting the flow lines are counted with signs determined by orientations of the descending manifolds.  In this paper we count the flow lines without any sign and do not use Morse homology.

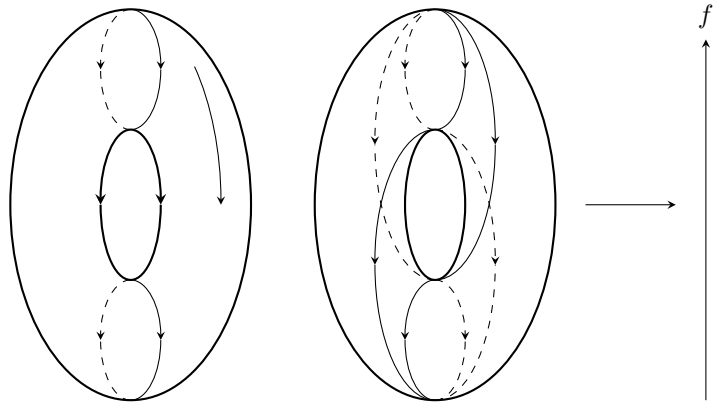
\begin{figure}
\begin{center}

\begin{tikzpicture}[scale = .4, >=stealth]
  \draw[thick] (0, 6.5) ellipse (4 and 6.5);
  \draw[thick, <->] (1, 6.5) arc (0:180:1 and 2.5);
  \draw[thick] (1, 6.5) arc (0:-180:1 and 2.5);

  \draw[dashed,->] (0, 4) arc (90:180:1 and 2);
  \draw[dashed] (0, 0) arc (270:180:1 and 2);
  \draw[->] (0, 4) arc (90:0:1 and 2);
  \draw (0, 0) arc (-90:0:1 and 2);
  \draw[dashed,->] (0, 13) arc (90:180:1 and 2);
  \draw[dashed] (0, 9) arc (270:180:1 and 2);
  \draw[->] (0, 13) arc (90:0:1 and 2);
  \draw (0, 9) arc (-90:0:1 and 2);

  \draw[<-] (3, 6.5) arc (0:45:3 and 6.5);
\end{tikzpicture}\hspace{20 pt}
\begin{tikzpicture}[scale = .4, auto, >=stealth]
  \draw[thick] (0, 6.5) ellipse (4 and 6.5);
  \draw[thick] (0, 6.5) ellipse (1 and 2.5);

  \draw[->] (0, 4) arc (90:180:1 and 2);
  \draw (0, 0) arc (270:180:1 and 2);
  \draw[dashed, ->] (0, 4) arc (90:0:1 and 2);
  \draw[dashed] (0, 0) arc (-90:0:1 and 2);
  \draw[dashed, ->] (0, 13) arc (90:180:1 and 2);
  \draw[dashed] (0, 9) arc (270:180:1 and 2);
  \draw[->] (0, 13) arc (90:0:1 and 2);
  \draw (0, 9) arc (-90:0:1 and 2);

  \draw[->] (0, 9) arc (90:180:2 and 4.5);
  \draw (0, 0) arc (270:180:2 and 4.5);
  \draw[dashed, ->] (0, 9) arc (90:0:2 and 4.5);
  \draw[dashed] (0, 0) arc (-90:0:2 and 4.5);
  \draw[dashed, ->] (0, 13) arc (90:180:2 and 4.5);
  \draw[dashed] (0, 4) arc (270:180:2 and 4.5);
  \draw[->] (0, 13) arc (90:0:2 and 4.5);
  \draw (0, 4) arc (-90:0:2 and 4.5);

  \path (5, 6.5) edge[->] node {} (8, 6.5);
  \node at (9, 12.75) {$f$};
  \path (9, 0) edge[->] node {} (9, 12);

\end{tikzpicture}
\end{center}
\caption{The height function of a standard torus stood on end, shown on the left, is not Morse-Smale because there is flow between the two index-$1$ critical points.  A slight perturbation of this function, shown on the right, is Morse-Smale and has eight $2$-part broken trajectories. (The picture is modeled on one in~\cite{Banyaga04}.)}\label{tori}
\end{figure}

Now we turn from the background to the specific theorems of Morse theory that are needed for the main theorem.  Let $(M, g)$ be a closed Riemannian manifold, and let $f : (M, g) \rightarrow \mathbb{R}$ be a Morse function satisfying Morse-Smale transversality.  We would like to say that the descending manifolds $\mathcal{D}(p)$ are the open cells of a nice CW-complex structure on $M$, in which the incidences between cells correspond somehow to breaking of trajectories and there is a triangulation that agrees with the CW-complex structure.  This wish is mostly true, but in the proof, the necessary analysis requires a special condition at the critical points.  

The pair $(f, g)$ is said to be \textbf{\emph{Euclidean}} if at every critical point $p$, there is a neighborhood with coordinates $(x_1, x_2) \in \mathbb{R}^{\ind(p)} \times \mathbb{R}^{n-\ind(p)}$ such that $g$ is equal to the standard Euclidean metric and $f$ is given by
\[f(x_1, x_2) = f(p) - \frac{1}{2}\abs{x_1}^2 + \frac{1}{2} \abs{x_2}^2.\]
(Here, $p$ has coordinates $(x_1, x_2) = (0, 0)$.)  We reduce the general case to the Euclidean case, and then work only in the Euclidean case; in~\cite{Wehrheim12}, Wehrheim justifies this reduction, which was stated by Franks in~\cite{Franks79} (or, see~\cite{Qin11} for a longer discussion).

\begin{theorem}[Remark 3.6 in~\cite{Wehrheim12}]\label{reduction}
Let $M$ be a closed manifold, and let $\Psi_s$ be the flow along some Morse-Smale negative gradient vector field.  Then there is a homeomorphism $h : M \rightarrow M$ such that $h \Psi_s h^{-1}$ is the flow along a Morse-Smale negative gradient vector field which in addition is Euclidean.
\end{theorem}

In particular, the number of $n$-part broken trajectories of $h \Psi_s h^{-1}$ is equal to the number of $n$-part broken trajectories of $\Psi_s$, so to prove the main theorem (Theorem~\ref{main-thm}) it suffices to consider the case where $(f, g)$ is Euclidean.  The goal of this section is the following lemma.

\begin{lemma}\label{morse-goal}
Let $(M, g)$ be a closed Riemannian manifold of dimension $n$, and suppose $f : (M, g) \rightarrow \mathbb{R}$ is Morse-Smale and Euclidean.  Let $\mathcal{S}$ be the partition into descending manifolds $\mathcal{D}(p)$.  Then $M$ is a CW-complex for which the generalized stratification by open cells is equal to $S$, and for each descending manifold $\mathcal{D}(p)$ of dimension $n$, the corresponding relative homology class $[\mathcal{D}(p)] \in H_n(M, M^{n-1})$ satisfies the bound
\[\Vert [\mathcal{D}(p)] \Vert_{\Delta, \ess}^{\mathcal{S}} \leq \#(n\text{-part broken trajectories beginning at }p).\]
\end{lemma}

We use Lizhen Qin's exposition~\cite{Qin10} as the reference for the Morse theory results; it presents self-contained proofs for a collection of Morse theory folk theorems.  In order to make $M$ into a CW-complex, the goal is to construct spaces $\overline{\mathcal{D}(p)}$ each homeomorphic to a closed ball with interior $\mathcal{D}(p)$, and maps $\overline{\mathcal{D}(p)} \rightarrow M$ for which the restriction to $\mathcal{D}(p)$ is the inclusion.  In the following discussion we fix notation necessary for describing the structure of $\overline{\mathcal{D}(p)}$.

For any two critical points $p$ and $q$, there is a smooth manifold $\mathcal{M}(p, q)$ which is the set of all unparametrized flow lines from $p$ to $q$.  Let $I = \{r_0, r_1, \ldots, r_{k+1}\}$ be a list of critical points, with 
\[\ind(r_0) > \ind(r_1) > \cdots > \ind(r_{k+1}).\]
The product manifolds $\mathcal{M}_I$ and $\mathcal{D}_I$ are defined by
\[\mathcal{M}_I = \prod_{i = 0}^k \mathcal{M}(r_i, r_{i+1}),\ \ \ \ \mathcal{D}_I = \mathcal{M}_I \times \mathcal{D}(r_{k+1}).\]
The length of $I$ is defined to be $\abs{I} = k$, which is two less than the number of critical points in $I$.  Then the spaces $\mathcal{M}(p, q)$ and $\mathcal{D}(p)$ can be compactified as
\[\overline{\mathcal{M}(p, q)} = \bigsqcup_{I} \mathcal{M}_I,\ \ \ \ \overline{\mathcal{D}(p)} = \bigsqcup_I \mathcal{D}_I,\]
where in $\overline{\mathcal{M}(p, q)}$ each $I$ is required to begin with $p$ and end with $q$, and in $\overline{\mathcal{D}(p)}$ each $I$ is required to begin with $p$.  Figure~\ref{morse-disk} depicts $\overline{\mathcal{D}(p)}$ in an example where $p$ is an index-$2$ critical point of a Morse-Smale function on $S^2$.

\begin{figure}
\begin{center}
\begin{tikzpicture}[scale = .4, vert/.style={circle, draw=black, fill=black, inner sep = 0pt, minimum size = 1mm}, >=stealth]
  \draw[thick] (8, 0) arc (0:-180:8 and 7);
  \draw[thick] (8, 0) arc (0:180:3 and 2);
  \draw[thick, ->] (2, 0) arc (0:-45:2 and 2);
  \draw[thick] (0, -2) arc (-90:-45:2 and 2);
  \draw[thick] (0, -2) arc (-90:-135:2 and 2);
  \draw[thick, ->] (-2, 0) arc (180:225:2 and 2);
  \draw[thick] (-2, 0) arc (0:180:3 and 3);

  \draw[dashed] (8, 0) arc (0:180:3 and 1);
  \draw (8, 0) arc (0:-180:3 and 1);
  \draw[dashed] (-2, 0) arc (0:180:3 and 1);
  \draw (-2, 0) arc (0:-180:3 and 1);

  \draw[->] (0, -2) arc (90:0:1 and 2.5);
  \draw (0, -7) arc (-90:0:1 and 2.5);
  \draw[dashed,->] (0, -2) arc (90:180:1 and 2.5);
  \draw[dashed] (0, -7) arc (270:180:1 and 2.5);

  \draw[->] (-5, 0) arc (180:240:5 and 7);

  \node[vert] at (-5, 3) [label=above: $p$] {};
  \node[vert] at (5, 2) [label=above: $a$] {};
  \node[vert] at (0, -2) [label=above: $b$] {};
  \node[vert] at (0, -7) [label=below: $c$] {};

  \draw[->] (10, -7)--(10,2);
  \node at (10, 3) {$f$};
\end{tikzpicture}\\ \vspace{20pt}
\begin{tikzpicture}[scale = .4, vert/.style={circle, draw=black, fill=black, inner sep = 0pt, minimum size = 1mm}, auto, >=stealth]
  \draw[thick] (0, 0) circle (4);

  \node[vert] at (0, 4) [label=above: ${\mathcal{M}(p, b)}\times {\mathcal{M}(b, c)} \times \mathcal{D}(c)$] {};
  \node[vert] at (0, -4) [label=below: ${\mathcal{M}(p, b)}\times {\mathcal{M}(b, c)} \times \mathcal{D}(c)$] {};
  \node at (4, 0) [label=right: ${\mathcal{M}(p, b)} \times \mathcal{D}(b)$] {};
  \node at (-4, 0) [label=left: ${\mathcal{M}(p, c)} \times \mathcal{D}(c)$] {};
  \node at (0, 0) {$\mathcal{D}(p)$};
\end{tikzpicture}
\end{center}
\caption{In this example the vertical coordinate gives a Morse-Smale function on $S^2$, shown in the upper picture.  The compactified disk $\overline{\mathcal{D}(p)}$, shown in the lower picture, consists of the following parts: the open $2$-cell $\mathcal{D}(p)$, the open segments $\mathcal{M}(p, b) \times \mathcal{D}(b)$ and $\mathcal{M}(p, c) \times \mathcal{D}(c)$, and the pair of points $\mathcal{M}(p, b) \times \mathcal{M}(b, c) \times \mathcal{D}(c)$.  Under the evaluation map $e : \overline{\mathcal{D}(p)} \rightarrow S^2$, the entire closed segment $\overline{\mathcal{M}(p, c)} \times \mathcal{D}(c)$, shown as the left-hand boundary of the disk, collapses to the point $c$.}\label{morse-disk}
\end{figure}
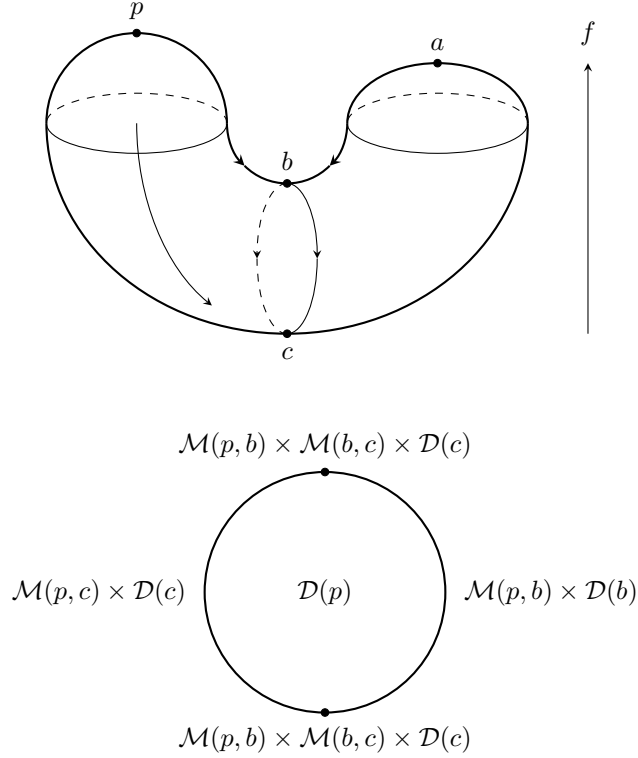

The theorems that describe the topology of these compactifications show that they are smooth manifolds with corners; a \textbf{\emph{smooth manifold with corners}} of dimension $n$ is defined by an atlas of open subsets of $[0, \infty)^n$ with smooth diffeomorphisms as transition functions.  The set of points that have exactly $k$ coordinates equal to zero (in any chart) is a manifold of dimension $n-k$, and we refer to it as the \textbf{\emph{codimension-$k$ part}}.  A \textbf{\emph{smooth manifold with faces}} is a smooth manifold with corners such that for every $k$, every point of the codimension-$k$ part belongs to the closures of $k$ different connected components of the codimension-$1$ part.

The description of $\overline{\mathcal{D}(p)}$ relies on the description of $\overline{\mathcal{M}(p, q)}$.

\begin{theorem}[Theorem 3.3 of~\cite{Qin10}]\label{3.3}
Let $(M, g)$ be a closed Riemannian manifold and suppose $f : (M, g) \rightarrow \mathbb{R}$ is Morse-Smale and Euclidean.  Then for every pair of critical points $(p, q)$, there is a smooth structure on $\overline{\mathcal{M}(p, q)}$ which satisfies the following properties:
\begin{itemize}
\item It is a compact smooth manifold with faces, and the codimension-$k$ part is $\bigsqcup_{\abs{I} = k}\mathcal{M}_I$.
\item The smooth structure is compatible with that of each $\mathcal{M}_I$.
\item For any critical point $r$ with $\ind(p) > \ind(r) > \ind(q)$, the concatenation map $\overline{\mathcal{M}(p, r)} \times \overline{\mathcal{M}(r, q)} \rightarrow \overline{\mathcal{M}(p, q)}$ is a smooth embedding.
\end{itemize}
\end{theorem}

\begin{theorem}[Theorem 3.4 of~\cite{Qin10}]\label{3.4}
Let $(M, g)$ be a closed Riemannian manifold and suppose $f : (M, g) \rightarrow \mathbb{R}$ is Morse-Smale and Euclidean.  Then for every critical point $p$, there is a smooth structure on $\overline{\mathcal{D}(p)}$ which satisfies the following properties:
\begin{itemize}
\item It is a compact smooth manifold with faces, and the codimension-$k$ part is $\bigsqcup_{\abs{I} = k-1}\mathcal{D}_I$.
\item The smooth structure is compatible with that of each $\mathcal{D}_I$.
\item For any critical point $r$ with $\ind(p) > \ind(r)$, the concatenation map $\overline{\mathcal{M}(p, r)} \times \overline{\mathcal{D}(r)} \rightarrow \overline{\mathcal{D}(p)}$ is a smooth embedding.
\item The evaluation map $e : \overline{\mathcal{D}(p)} \rightarrow M$ is smooth, where the restriction of $e$ on $\mathcal{D}_I = \mathcal{M}_I \times \mathcal{D}(r_{k+1})$ is the coordinate projection onto $\mathcal{D}(r_{k+1}) \subseteq M$.
\end{itemize}
\end{theorem}

To make $M$ into a CW-complex, what remains is to show that $\overline{\mathcal{D}(p)}$ with this topology is homeomorphic to a closed ball.

\begin{theorem}[Theorem 3.7 of~\cite{Qin10}]\label{ball}
Let $(M, g)$ be a closed Riemannian manifold and suppose $f : (M, g) \rightarrow \mathbb{R}$ is Morse-Smale and Euclidean.  Let $D^k$ denote the closed ball of dimension $k$, and let $S^{k-1}$ denote its boundary.  Then for every critical point $p$ there is a homeomorphism between the pair $(\overline{\mathcal{D}(p)}, \del \overline{\mathcal{D}(p)})$ and the pair $(D^{\ind(p)}, S^{\ind(p) - 1})$.
\end{theorem}

In order to prove the main lemma of this section, Lemma~\ref{morse-goal}, we want to find a triangulation of each $n$-dimensional cell $\overline{\mathcal{D}(p)}$ that satisfies the \emph{cellular}, \emph{order}, \emph{internality}, and \emph{loop} properties and does not have too many essential simplices.  Let $\mathcal{S}_1$ denote the stratification of $\overline{\mathcal{D}(p)}$ consisting of the connected components of the codimension-$k$ part for each $k$.  We will first find a finite triangulation of $\overline{\mathcal{D}(p)}$ with the \emph{cellular} property with respect to $\mathcal{S}_1$; an iterated barycentric subdivision of this triangulation satisfies all four properties by Sublemma~\ref{bary}, and the next lemma counts the number of essential simplices with respect to $\mathcal{S}_1$.  Figure~\ref{triangulated-disk} depicts an example triangulation of the disk $\overline{\mathcal{D}(p)}$ from Figure~\ref{morse-disk}, labeled with its essential simplices.

\begin{lemma}\label{corner}
Let $X = [0, \infty)^n$, and let $\mathcal{S}_1$ be the stratification of $X$ consisting of the connected components of the codimension-$k$ part for each $k$.  Let $c_1$ be a finite triangulation of the closure of a neighborhood of $0$ in $X$, satisfying the \emph{cellular} property with respect to $\mathcal{S}_1$, and let $c_2$ denote the first barycentric subdivision of $c_1$.  Then the map sending each simplex of $c_2$ to the list of strata of its vertices induces a bijection
\[\text{essential simplices of }c_2 \leftrightarrow \text{totally ordered chains of }n+1\text{ strata}.\]
\end{lemma}

\begin{proof}
The proof is by induction on $n$.  The base case is $n = 0$, in which case $X$ is one point.  For $n > 0$, we apply the inductive hypothesis to each $(n-1)$-dimensional face of $X$ with the induced triangulation.  The resulting essential $(n-1)$-simplices of $c_2 \cap \del X$ are in bijection with the totally ordered chains of $n$ strata in $\del X$.  There is only one $n$-dimensional stratum, which we append to the totally ordered chains of $n$ strata in $\del X$ to get the totally ordered chains of $n+1$ strata in $X$.  An $n$-simplex of $X$ is essential if and only if it has one vertex in the interior of $X$ and the opposite face is an essential $(n-1)$-simplex of $\del X$; there is exactly one such $n$-simplex of $c_2$ for each $(n-1)$-simplex of $c_2 \cap \del X$.  Thus the essential $n$-simplices are in bijection with the totally ordered chains of $n+1$ strata in $X$.
\end{proof}

We use the following easy lemma to show that when counting totally ordered chains of $n+1$ strata, it is equivalent to count them globally on $\overline{\mathcal{D}(p)}$ or locally near the vertices.

\begin{lemma}\label{faces}
Let $D$ be an $n$-dimensional compact smooth manifold with faces.  Let $v$ be a vertex of the codimension-$n$ part, and let $U$ be a neighborhood of $v$ corresponding to a ball around $0$ in $[0, \infty)^n$.  Then for every $k$, the connected components of the codimension-$k$ part of $U$ all come from different connected components of the codimension-$k$ part of $D$.
\end{lemma}

\begin{proof}
Every point $x$ in the codimension-$k$ part can be labeled by which $k$ connected components of the codimension-$1$ part of $D$ have closure containing $x$.  This labeling is constant as $x$ varies within a connected component of the codimension-$k$ part of $D$.  Coming together at $v$, there are $n$ different connected components of the codimension-$1$ part of $D$, so in $U$ every connected component of the codimension-$k$ part gets a different label.  Thus none of them can be part of the same connected component of the codimension-$k$ part of $D$.
\end{proof}

The final ingredient needed to prove Lemma~\ref{morse-goal}, which is the goal of this section, is the following lemma.  It relates the stratification $\mathcal{S}_1$, which comes from the structure of $\overline{\mathcal{D}(p)}$ as a manifold with corners, to the generalized stratification of $M$ corresponding to its CW structure.  Specifically, consider the evaluation map $e : \overline{\mathcal{D}(p)} \rightarrow M$ from Theorem~\ref{3.4}, and let $\mathcal{S}_2$ be the generalized stratification of $\overline{\mathcal{D}(p)}$ consisting of the space $e^{-1}(\mathcal{D}(r)) = \overline{\mathcal{M}(p, r)} \times \mathcal{D}(r)$ for each critical point $r$.  Note that every stratum of $\mathcal{S}_2$ is a union of strata from $\mathcal{S}_1$.  Figure~\ref{triangulated-disk} illustrates the fact that some simplices are essential with respect to $\mathcal{S}_1$ but are not essential with respect to $\mathcal{S}_2$.

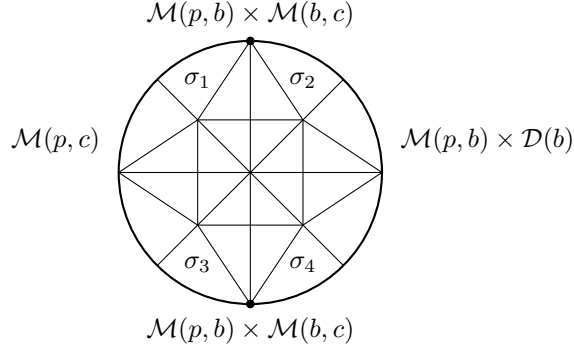
\begin{figure}
\begin{center}
\begin{tikzpicture}[scale = .35, vert/.style={circle, draw=black, fill=black, inner sep = 0pt, minimum size = 1mm}]
  \draw[thick] (0, 0) circle (5);

  \node[vert] at (0, 5) [label=above: ${\mathcal{M}(p, b)}\times {\mathcal{M}(b, c)}$] {};
  \node[vert] at (0, -5) [label=below: ${\mathcal{M}(p, b)}\times {\mathcal{M}(b, c)}$] {};
  \node at (5, 0) [label=above right: ${\mathcal{M}(p, b)} \times \mathcal{D}(b)$] {};
  \node at (-5, 0) [label=above left: ${\mathcal{M}(p, c)}$] {};

  \draw (-5, 0)--(5, 0);
  \draw (0, -5)--(0, 5);
  \draw (-3.536, -3.536)--(3.536, 3.536);
  \draw (-3.536, 3.536)--(3.536, -3.536);
  \draw (-2,-2) rectangle (2,2);
  \draw (-2,-2)--(-5, 0)--(-2,2)--(0,5)--(2,2)--(5,0)--(2,-2)--(0,-5)--(-2,-2);

  \node at (-2, 3.5) {$\sigma_1$};
  \node at (2,3.5) {$\sigma_2$};
  \node at (-2, -3.5) {$\sigma_3$};
  \node at (2, -3.5) {$\sigma_4$};
\end{tikzpicture}
\end{center}
\caption{In this stratified triangulation $c_2$ of the disk $\overline{\mathcal{D}(p)}$ from Figure~\ref{morse-disk}, the essential simplices with respect to stratification $\mathcal{S}_1$ (which consists of an open $2$-cell, two open $1$-cells, and two $0$-cells) are $\sigma_1$, $\sigma_2$, $\sigma_3$, and $\sigma_4$.  The essential simplices with respect to generalized stratification $\mathcal{S}_2$ (which consists of an open $2$-cell, an open $1$-cell shown on the right boundary, and a closed $1$-cell shown on the left boundary) are $\sigma_2$ and $\sigma_4$ only.}\label{triangulated-disk}
\end{figure}

\begin{lemma}\label{gray}
Let $(M, g)$ be a closed Riemannian manifold and suppose $f : (M, g) \rightarrow \mathbb{R}$ is Morse-Smale and Euclidean.  Let $p$ be any critical point.  Let $\mathcal{S}_1$ be the stratification of $\overline{\mathcal{D}(p)}$ consisting of the connected components of the codimension-$k$ part for each $k$, and let $\mathcal{S}_2$ be the stratification of $\overline{\mathcal{D}(p)}$ consisting of the space $\overline{\mathcal{M}(p, r)} \times \mathcal{D}(r)$ for each critical point $r$.  Then the number of totally ordered chains of $\ind(p)+1$ strata from $\mathcal{S}_1$ such that those strata belong to $\ind(p)+1$ different strata from $\mathcal{S}_2$ is the number of $\ind(p)$-part broken trajectories starting at $p$.
\end{lemma}

\begin{proof}
Roughly, a totally ordered chain of $\ind(p) + 1$ strata from $\mathcal{S}_1$ corresponds to an increasing sequence of sets of critical points at which the flow lines from $p$ may break.  For such a set of critical points, the corresponding stratum of $\mathcal{S}_2$ is determined by the critical point of least index.  Thus, in order to get $\ind(p) + 1$ different strata from $\mathcal{S}_2$, the critical points must be added in decreasing order of index.  The number of ways to make such a chain of $\mathcal{S}_1$-strata is the number of $\ind(p)$-part broken trajectories.

The proof is by induction on $\ind(p)$.  The base case is $\ind(p) = 0$, in which case $\overline{\mathcal{D}(p)}$ is one point.  For $\ind(p) > 0$, every totally ordered chain of $\ind(p) + 1$ strata from $\mathcal{S}_1$ consists of the maximal stratum $\mathcal{D}(p)$ along with some totally ordered chain of $\ind(p)$ strata with maximum equal to some codimension-$1$ stratum.  We know that the codimension-$1$ part of $\overline{\mathcal{D}(p)}$ consists of various $\mathcal{M}(p, r) \times \mathcal{D}(r)$, where $r$ is a critical point with $\ind(r) < \ind(p)$.

If $\ind(r) = \ind(p) - 1$, then $\mathcal{M}(p, r)$ is a finite set.  We know that concatenation gives an embedding $\overline{\mathcal{M}(p, r)} \times \overline{\mathcal{D}(r)} \rightarrow \overline{\mathcal{D}(p)}$, by Theorem~\ref{3.4}.  The space $\overline{\mathcal{M}(p, r)} \times \overline{\mathcal{D}(r)} = \mathcal{M}(p, r) \times \overline{\mathcal{D}(r)}$ is a finite disjoint union of copies of $\overline{\mathcal{D}(r)}$, so any totally ordered chain of strata with maximum equal to some stratum in $\mathcal{M}(p, r) \times \mathcal{D}(r)$ is completely contained in one of the copies of $\overline{\mathcal{D}(r)}$.  We apply the inductive hypothesis to each $\overline{\mathcal{D}(r)}$.  Then, the $\mathcal{S}_2$-stratum $\mathcal{D}(p)$ is different from every $\mathcal{S}_2$-stratum intersecting any $\overline{\mathcal{D}(r)}$, so the number of totally ordered chains of $\ind(p) + 1$ strata from $\mathcal{S}_1$, with maximum $\mathcal{D}(p)$ and next stratum in $\mathcal{M}(p, r) \times \mathcal{D}(r)$, such that these strata belong to $\ind(p) + 1$ different strata from $\mathcal{S}_2$, is the number of $(\ind(p) + 1)$-part broken trajectories starting with $p$ and then $r$.

If, on the other hand, $\ind(r) < \ind(p) - 1$, then as before we have an embedding $\overline{\mathcal{M}(p, r)} \times \overline{\mathcal{D}(r)} \rightarrow \overline{\mathcal{D}(p)}$.  We claim that in $\overline{\mathcal{M}(p, r)} \times \overline{\mathcal{D}(r)}$ there are no totally ordered chains of $\ind(p)$ strata from $\mathcal{S}_1$ such that these strata belong to $\ind(p)$ different strata from $\mathcal{S}_2$.  First we observe that if two strata $S$ and $S'$ of $\mathcal{S}_1$ are related by $S \preceq S'$, then the strata of $\mathcal{S}_2$ containing $S$ and $S'$ satisfy the same relation.  Thus for every totally ordered chain of strata from $\mathcal{S}_1$, the corresponding strata from $\mathcal{S}_2$ are also totally ordered.  We also observe that for any two critical points $q_1$ and $q_2$ of the same index, the corresponding $\mathcal{S}_2$-strata $\overline{\mathcal{M}(p, q_1)}\times \mathcal{D}(q_1)$ and $\overline{\mathcal{M}(p, q_2)}\times \mathcal{D}(q_2)$ are incomparable.  This implies that the maximum possible length of a totally ordered chain of $\mathcal{S}_2$-strata intersecting $\overline{\mathcal{M}(p, r)} \times \overline{\mathcal{D}(r)}$ is $\ind(r) + 1$, which is less than $\ind(p)$.

Thus, back in $\overline{\mathcal{D}(p)}$, for every totally ordered chain of $\ind(p) + 1$ strata from $\mathcal{S}_1$ such that those strata belong to $\ind(p) + 1$ different strata from $\mathcal{S}_2$, the next-to-maximum stratum is in $\mathcal{M}(p, r) \times \mathcal{D}(r)$ for some critical point $r$ with $\ind(r) = \ind(p) - 1$.  Adding up the number of ways to do this for all such $r$, the total is the number of $\ind(p)$-part broken trajectories starting at $p$.
\end{proof}

\begin{proof}[Proof of Lemma~\ref{morse-goal}]
Theorems~\ref{3.4} and~\ref{ball} imply that $M$ is a CW-complex with the descending manifolds as open cells.  As a smooth manifold with corners, each $\overline{\mathcal{D}(p)}$ is a Whitney stratified space and therefore has a stratified triangulation (see~\cite{Goresky78} and \cite{Goresky88}; the former includes a definition of stratified triangulation, and the latter includes a definition of Whitney stratified space).  The stratified triangulation satisfies the \emph{cellular} property with respect to $\mathcal{S}_1$, and there are finitely many simplices because $\overline{\mathcal{D}(p)}$ is compact.  Let $c_{\rel}(p)$ be an iterated barycentric subdivision of this original triangulation.  

By Sublemma~\ref{bary}, the relative $(\overline{\mathcal{D}(p)}, \del \overline{\mathcal{D}(p)})$ cycle $c_{\rel}(p)$ satisfies the \emph{cellular}, \emph{order}, \emph{internality}, and \emph{loop} properties with respect to $\mathcal{S}_1$, so it also satisfies these four properties with respect to $\mathcal{S}_2$, or equivalently, as a relative $(M, M^{n-1})$ cycle with respect to $\mathcal{S}$.  By Lemmas~\ref{corner} and~\ref{faces}, the triangulation $c_{\rel}(p)$ has one essential simplex with respect to $\mathcal{S}_1$ for each totally ordered chain of $n+1$ strata from $\mathcal{S}_1$, so by Lemma~\ref{gray}, $c_{\rel}(p)$ has one essential simplex with respect to $\mathcal{S}_2$ (or equivalently $\mathcal{S}$) for each $n$-part broken trajectory beginning at $p$.  Thus we have
\[\Vert [\mathcal{D}(p)] \Vert_{\Delta, \ess}^\mathcal{S} \leq \Vert c_{\rel}(p) \Vert_{\Delta, \ess}^\mathcal{S} = \#(n\text{-part broken trajectories beginning at }p).\]
\end{proof}

\section{Main theorem}\label{end}

The main theorem, Theorem~\ref{main-thm}, is an immediate corollary of the following theorem.  This more general theorem bounds the number of $n$-part broken trajectories in terms of the simplicial volume, instead of the hyperbolic volume, so it applies to all closed manifolds with nonzero simplicial volume.

\begin{theorem}\label{general}
Let $M$ be a closed, oriented manifold of dimension $n$.  Let $g$ be a Riemannian metric on $M$, and let $f : (M, g) \rightarrow \mathbb{R}$ be a Morse function satisfying Morse-Smale transversality.  Let $Z$ be any topological space such that its universal cover is contractible, and let $\alpha: M \rightarrow Z$ be a continuous map.  Then the fundamental homology class $[M]$ has $\alpha$-image in $H_n(Z)$ with simplicial norm satisfying the bound
\[\Vert \alpha_*[M]\Vert_{\Delta} \leq \#(n\text{-part broken trajectories of }{-\nabla f}).\]
\end{theorem}

\begin{proof}
By Theorem~\ref{reduction} we may assume that $(f, g)$ is Euclidean.  Let $\mathcal{S}$ denote the generalized stratification of $M$ corresponding to the CW structure from the descending manifolds.  The relative class $[M]_{\rel} \in H_n(M, M^{n-1})$ corresponding to $[M]$ is equal to the sum of $n$-cells
\[[M]_{\rel} = \sum_{\ind(p) = n} [\mathcal{D}(p)],\]
and so by Lemma~\ref{morse-goal} its essential stratified simplicial norm satisfies the bound
\[\Vert [M]_{\rel} \Vert_{\Delta, \ess}^{\mathcal{S}} \leq \sum_{\ind(p) = n} \Vert [\mathcal{D}(p)] \Vert_{\Delta, \ess}^{\mathcal{S}} \leq \#(n\text{-part broken trajectories}).\]
Applying Lemma~\ref{simplicial-goal}, we have the inequality
\[\Vert \alpha_*[M]\Vert_\Delta \leq \Vert [M]_{\rel} \Vert_{\Delta, \ess}^{\mathcal{S}},\]
which proves the theorem.
\end{proof}

\begin{proof}[Proof of Theorem~\ref{main-thm}]
Suppose $M$ is hyperbolic.  For $Z$ we take $M$, and for $\alpha : M \rightarrow Z$ we take the identity map on $M$.  Combining Theorem~\ref{general} with the formula relating simplicial norm to hyperbolic volume, we obtain
\[\frac{\Vol(M, \hyp)}{\Vol \Delta^n} = \Vert [M] \Vert_{\Delta} \leq \#(n\text{-part broken trajectories}).\]
\end{proof}

\bibliography{bib-morse}{}
\bibliographystyle{amsalpha}
\end{document}